\documentclass[a4paper, 12pt]{amsart}
\usepackage[a4paper,hmargin=2.5cm,vmargin=2.5cm]{geometry}
\usepackage[T2A]{fontenc}
\usepackage[utf8]{inputenc}

\usepackage[russian,english]{babel}
\usepackage{amsmath,amssymb,amsthm,amsfonts,mathrsfs,amscd}
\usepackage[linkcolor=blue,citecolor=blue,colorlinks=true,urlcolor=blue]{hyperref}
\usepackage{comment}
\usepackage{graphics}

\numberwithin{equation}{section}
\newtheorem{theorem}{Theorem}[section]
\newtheorem{lemma}[theorem]{Lemma}
\newtheorem{propos}[theorem]{Proposition}
\theoremstyle{definition}
\newtheorem{definition}[theorem]{Definition}

\newtheorem{proof}{Proof}
\let\origendproof\endproof
\def\endproof{\unskip\nobreak\hskip5pt plus 1fill$\square$\origendproof}

\newtheorem{sled}[theorem]{Corollary}

\newtheorem{rem}[theorem]{Remark}

\def\PGL{\mathop{\rm PGL}\nolimits}
\def\GL{\mathop{\rm GL}\nolimits}
\def\Cr{\mathop{\rm Cr}\nolimits}

\def\Aut{\mathop{\rm Aut}\nolimits}

\begin{document}
\title{Jordan constants of Cremona group of rank 2 over fields of characteristic zero}
\author{A.\,V.~Zaitsev}
\address{National research university ''Higher school of economics'', Laboratory of algebraic geometry, 6 Usacheva str., Moscow, 119048, Russia}
\email{\href{alvlzaitsev1@gmail.com}{alvlzaitsev1@gmail.com}}

\maketitle

\begin{abstract}
    In this paper we compute the Jordan constants of the Cremona group of rank two over all fields of characteristic zero.
\end{abstract}

\section{Introduction}

Groups of birational automorphisms of algebraic varieties are often complicated. For example, finding a set of generators of such groups may be a nontrivial task already in dimension 2. As one of the possible ways to study such complicated groups, one can try to describe their finite subgroups. In particular, one can study the Jordan property of infinite groups.

\begin{definition}[{\cite[Definition 2.1]{Pop}}]\label{def: J.} 
Let $G$ be a finite group. The Jordan constant $J(G)$ of $G$ is the smallest index of a normal abelian subgroup in $G$. Let $\Gamma$ be an arbitrary group. Then $\Gamma$ is called \textit{Jordan} if the value $$ J(\Gamma) = \sup\limits_{G\subseteq \Gamma,\ |G|<\infty }(J(G))$$ is finite. In this case the number $J(\Gamma)$ is called the \textit{Jordan constant} of the group $\Gamma$.
\end{definition}

At the end of the 19th century, Camille Jordan proved that complete linear groups~$\text{GL}_n(K)$ over a field $K$ of characteristic zero are Jordan, see~\cite[\S 40]{Jord} or~\mbox{\cite[Theorem 36.13]{CR}}. The Jordan constants of these groups were computed over algebraically closed fields, much later, in 2004, in the work~\cite{Coll}. As a corollary, all linear algebraic groups are Jordan.

We are also interested in the question of Jordan property of another class of groups of geometric origin, namely the groups of birational automorphisms of projective spaces. These groups are called Cremona groups and are denoted by $\Cr_n(K)$, where $n$ is the dimension of the projective space (it is called the rank of the Cremona group), and $K$~is a base field. Groups $\Cr_n(K)$ are Jordan over fields of characteristic zero. In dimension two, this was proved in the paper~\cite[Theor\'em\`e 3.1]{Serre}, and in arbitrary dimension a similar result was obtained in~\mbox{\cite[Theorem 1.8]{PrShr}} and~\hbox{\cite[Theorem 1.1]{Birkar}}. The situation with Jordan constants of these groups is more complicated than for linear groups. For now, the exact values of the Jordan constants of the group of birational automorphisms of the projective plane have been found over algebraically closed fields of characteristic zero, as well as over fields of real and rational numbers. Namely, the following theorem holds.

\begin{theorem}[{\cite[Theorems 1.9, 1.10, 1.11]{Yas}}]\label{theo: Egor}
    Let $K$ be an algebraically closed field of characteristic $0$. Then $$J(\Cr_2(K)) = 7200, \quad J(\Cr_2(\mathbb{R})) = J(\Cr_2(\mathbb{Q})) = 120.$$
\end{theorem}

In dimension 3, the exact values are unknown even in the case of algebraically closed field. But there are upper bounds, see~\mbox{\cite[Theorem 1.2.4]{PrShr2}}. It is also worth noting that the groups of birational automorphisms of projective spaces over algebraically closed fields of positive characteristic are not Jordan, since already automorphism groups are not Jordan (see for example ~\cite[\S 1]{Chen-Shramov}). But for finite field the group of birational automorphisms of the projective plane is Jordan, and even its Jordan constant is computed, see~\cite{PrShr3} and~\cite{Vikulova}.

In this paper we compute the Jordan constants of the birational automorphism group of the projective plane over all fields of characteristic zero. Namely, we prove the following theorem.

\begin{theorem}\label{theo: J(Cr_2)}
    Let $K$ be a field of characteristic $0$. The Jordan constant of the group $\Cr_2(K)$ can attain the following values: $7200$, $168$ or $120$. Moreover,
    \begin{enumerate}
        \item $J(\Cr_2(K)) = 7200$ if and only if $\sqrt{5} \in K$, and $-1$ is a sum of two squares in $K$;
        \item $J(\Cr_2(K)) = 168$ if and only if $\sqrt{5} \not\in K$ and $\sqrt{-7} \in K$;
        \item $J(\Cr_2(K)) = 120$ if and only if one of the two conditions is satisfied:
        \begin{itemize}
            \item $\sqrt{5} \in K$, and $-1$ is not a sum of two squares in $K$; 
            \item $\sqrt{5} \not\in K$, and $\sqrt{-7} \not\in K$.
            \end{itemize}
    \end{enumerate}
\end{theorem}

We also give examples of fields for each value of the Jordan constant from Theorem~\ref{theo: J(Cr_2)}. Everywhere in this work, $\omega$ will denote a primitive root of unity of degree~3.

\begin{sled}\label{sled: Cr(k), Cr(R), Cr(Q)...} 
All values of the Jordan constant from Theorem~\ref{theo: J(Cr_2)} are attained:
$$J(\Cr_2(\mathbb{C})) = J(\Cr_2(\mathbb{Q}(\sqrt{5}, \sqrt{-7}))) = J(\Cr_2(\mathbb{Q}(\sqrt{5}, \omega))) = 7200,$$ $$ J(\Cr_2(\mathbb{Q}(\sqrt{-7}))) = 168, \quad J(\Cr_2(\mathbb{R})) =  J(\Cr_2(\mathbb{Q}(\sqrt{5})) = J(\Cr_2(\mathbb{Q})) = 120.$$
\end{sled}

The plan of the paper is as follows. In Section~\ref{sect: preliminaries} we collect some auxiliary statements. In Section~\ref{sect: conic bundles} we estimate the Jordan constants of automorphism groups of surfaces having a conic bundle structure over $\mathbb{P}^1$. In Section~\ref{sect: del Pezzo} we estimate the Jordan constants of automorphism groups of del Pezzo surfaces. Finally, in Section~\ref{sect: Jordan constant of Cremona}, taking into account all previous estimates, we prove Theorem~\ref{theo: J(Cr_2)} and Corollary~\ref{sled: Cr(k), Cr(R), Cr(Q)...}.

We will use the following notation. By $D_{2n}$ we denote the dihedral group of order~${2n}$. By $\overline{K}$ we denote the algebraic closure of the field $K$. If $K \subset L$ is an extension of fields, and $X$ is a variety over $K$, then by $X_L$ we denote the extension of scalars of $X$ to $L$, and we denote the extension of scalars of $X$ to $\overline{K}$ by $\overline{X}$. Zariski tangent space to a variety~$X$ at a point $P$ we denote by $T_PX$. Let $G$ be a group, and $g$, $g' \in G$, then by $\langle g,g' \rangle$ we denote thе subgroup, generated by elements $g$ и~$g'$. Let $\pi: X \xrightarrow{} B$ be a morphism of algebraic varieties; then by $\Aut(X,\pi)$ we denote the group of automorphisms compatible with the morphism $\pi$, that is, the group of such automorphisms $g\in\Aut(X)$ that there exists~$g'\in\Aut(B)$ for which $\pi\circ g = g'\circ\pi$.

I would like to thank my advisor Constantin Shramov for stating the problem, useful discussions
and constant attention to this work. I also want to thank Andrey Trepalin for useful discussions.

\section{Preliminaries}\label{sect: preliminaries}

In this section, we collect some (well known) general statements that will be convenient to refer later.

Firstly, let us recall the standard definition.

\begin{definition}\label{def: char sub}
A subgroup $H$ of a group $G$ is called a \textit{characteristic} subgroup if for every automorphism $\varphi$ of $G$, one has $\varphi(H) = H$.
\end{definition}

The following theorem is useful for estimating Jordan constants of finite groups.

\begin{theorem}[{see~\cite[Theorem 1.41]{Isaacs}}]\label{theo: weak jordan constant} 
Let $G$ be a finite group, and $A$ be its abelian subgroup. Then there is a characteristic abelian subgroup $N$ in $G$ such that $$[G:N]\leqslant [G:A]^2.$$
\end{theorem}

The proofs of the following two propositions can be found in~\cite{Zai}.

\begin{propos}[{\cite[Corollary 2.9]{Zai}}]\label{prop: ext of S_n}
Let $n$ be an integer number, $n \geqslant 3$, $n \neq 6$. Let $A$ be a finite abelian group. Then any group $G$ included in the exact sequence $$1\xrightarrow{}\mathfrak{S}_n\xrightarrow{} G \xrightarrow{} A\xrightarrow{} 0,$$ is isomorphic to the direct product of $\mathfrak{S}_n\times A$.
\end{propos}

\begin{propos}[{\cite[Proposition 2.11]{Zai}}]\label{prop: A_4 x Zm}
Let $n$ and $m$ be positive integers, $n\geqslant 4$, $n\neq 6$. Let the group $G$ be included in the exact sequence $$1 \xrightarrow {} \mathfrak {A} _n \xrightarrow {} G\xrightarrow {} \mathbb {Z} / m \mathbb{Z} \xrightarrow{} 0.$$ If $m = 2k + 1$, then $G$ is isomorphic to $\mathfrak{A}_n\times\mathbb{Z}/m\mathbb{Z}$. If $m = 2k$, then either $G$ is isomorphic to~$\mathfrak{A}_n\times\mathbb{Z}/m\mathbb{Z}$, or at least contains a normal subgroup isomorphic to $\mathfrak{A}_n\times\mathbb{Z}/k\mathbb{Z}$.
\end{propos}

In this paper, we will need some information about finite subgroups in the group~$\PGL_2(K)$, where $K$ is an arbitrary field of characteristic zero. For our purposes, the following proposition will be convenient.

\begin{propos}[{cf.~\cite[Proposition 1.1]{Bea}}]\label{prop:Beauville}
Let $K$ be a field of characteristic $0$. Then the group $\PGL_2(K)$ can contain only finite subgroups isomorphic to~$\mathbb{Z}/n\mathbb{Z}$, $D_{2n}\; 
(\text{for}\;n\geqslant 2)$, $\mathfrak{A}_4$, $\mathfrak{S}_4$ and $\mathfrak{A}_5$. Moreover, $\PGL_2(K)$ contains $\mathfrak{A}_5$ if and only if $-1$ is a sum of two squares in $K$, and $K$ contains $\sqrt{5}$.
\end{propos}

Jordan constants of the group $\GL_2(K)$ are computed over all fields of characteristic zero in paper~\cite{Hu}, but for our purposes we limit ourselves to the following estimate.

\begin{lemma}\label{lemma:J(Gl_2 <= 60)}
     Let $K$ be a field of characteristic $0$. Then $$J(\GL_2(K)) \leqslant 60.$$
\end{lemma}
\begin{proof}
It is easy to see that any finite subgroup in $\GL_2(K)$ is a central extension of a finite subgroup in $\PGL_2(K)$ by a finite cyclic group. Let $G\subset\GL_2(K)$ be a finite subgroup, then it is included into the exact sequence $$0\xrightarrow{}\mathbb{Z}/r\mathbb{Z}\xrightarrow{} G\xrightarrow{} G'\xrightarrow{} 1,$$ where $G'$ is a finite subgroup in $\PGL_2(K)$. By Proposition~\ref{prop:Beauville} the group $G'$ is either cyclic, or dihedral, or isomorphic to $\mathfrak{A}_4$, $\mathfrak{S}_4$, or $\mathfrak{A}_5$.

If $G'$ is isomorphic to $\mathfrak{A}_4$, $\mathfrak{S}_4$ or $\mathfrak{A}_5$, then the subgroup $\mathbb{Z}/r\mathbb{Z}\subset G$ is a normal abelian subgroup of index at most $60$, hence $J(G) \leqslant 60$.

If $G'$ is cyclic, then the group $G$ is the central extension of a finite cyclic group by a finite cyclic group, hence $G$ is abelian, and $J(G) = 1$.

If $G'$ is dihedral, then we can take the normal cyclic subgroup $\mathbb{Z}/n\mathbb{Z}\subset G'$ of index~$2$. The preimage of $\mathbb{Z}/n\mathbb{Z}$ is a normal abelian subgroup in $G$ of index~$2$, hence~$J(G)\leqslant~2$.

Thus, all cases are considered, and the lemma is proved. \end{proof}

The following theorem about the action of a finite group on a tangent space of a variety at fixed point is well known. However, in different sources it is formulated in different generality, so we also give a proof of it.

\begin{theorem}\label{theo: proof of action with fixed point}
Let $X$ be an irreducible algebraic variety over a field $K$ of characteristic~$0$. Let $G$ be a finite group acting faithfully on $X$ with a fixed~$K$-point $P$. Then the natural homomorphism $$d:G\xrightarrow{}\GL(T_PX)$$ is injective.
\end{theorem}

\begin{proof}
Suppose $d$ is not injective. Then, replacing $G$ by the kernel of the homomorphism~$d$, we can assume that $G$ acts trivially on the tangent space. Hence $G$ acts trivially on the cotangent space~$\mathfrak{m}_P/\mathfrak{m}^2_P$, where~$\mathfrak{m}_P\subset\mathcal{O}_P$ is the maximal ideal of the local ring of the point $P$. The morphism $$\mathcal{O}_P\xrightarrow{}\mathcal{O}_P/\mathfrak{m}_P\simeq K$$ has the natural section, hence we have $G$-invariant decomposition~$\mathcal{O}_P\simeq\mathfrak{m}_P \oplus K$ into a direct sum of vector subspaces.

Consider a $G$-invariant filtration $$\mathfrak{m}_P\supset\mathfrak{m}^2_P\supset\mathfrak{m}^3_P\supset\dots$$According to Nakayama's lemma, the ideal $\mathfrak{m}_P$ is generated by elements whose images form a basis of the space $\mathfrak{m}_P/\mathfrak{m}^2_P$. Hence, $G$ acts trivially on vector spaces~$\mathfrak{m}^n_P/\mathfrak{m}^{n+1}_P$ for all integers $n\geqslant 1$. According to Maschke's theorem, any representation of a finite group $G$ over a field of characteristic zero is completely reducible. Hence we have an isomorphism of representations of the group $G$: $$\mathfrak{m}_P\simeq\bigoplus\limits_{n=1}^\infty\mathfrak{m}^n_P/\mathfrak{m}^{n+1}_P.$$ Therefore, $G$ acts trivially on $\mathfrak{m}_P$ and on $\mathcal{O}_P$.

Let $U\subset X$ be an affine neighborhood of the point $P$. Take an arbitrary element $g\in G$, then $U' = g(U)$ is also an affine neighborhood of the point $P$. Let $R$ and $R'$ be the coordinate rings of neighborhoods $U$ and $U'$, respectively. Then $R$ and $R'$ are subrings in the local ring $\mathcal{O}_P$, such that $g^*R'= R$. Since $g^*$ is trivial automorphism of the ring $\mathcal{O}_P$, then $R = R'$ and $g$ trivially acts on $R$. Therefore $U = U'$ and $g$ trivially acts on $U$. Using the irreducibility of $X$, we conclude that $g$ acts on $X$ by a trivial automorphism. So $g$ is a neutral element, and the group $G$ is trivial.
\end{proof}

The following corollary will be useful for estimating Jordan constants of groups acting on surfaces with fixed points.

\begin{sled}\label{lemma: action with fixed point}
Let $S$ be a smooth surface over a field $K$. Let $\Gamma$ be a group acting faithfully on $S$ with a fixed $\overline{K}$-point $P$. Then $J(\Gamma) \leqslant 60$.
\end{sled}

\begin{proof}
Let $G\subset\Gamma$ be a finite subgroup. Then the action of~$G$ on $S$ induces the action of~$G$ on the tangent space $T_{P}S$. By Theorem~\ref{theo: proof of action with fixed point} this action is faithful, that is, the group~$G$ embeds in the group~$\GL(T_PS)\simeq\GL_2\left(\overline{K}\right)$. Using Lemma~\ref{lemma:J(Gl_2 <= 60)}, we get inequality $$J(G) \leqslant J\left(\GL_2\left(\overline{K}\right)\right) \leqslant 60.$$ Hence, $J(\Gamma) \leqslant 60$.
\end{proof}

We will also need the following remark.

\begin{rem}\label{rem: -1 = a^2 + b^2 in special fields }
In the field $\mathbb{Q}(\sqrt{5}, \sqrt{-7})$ and in the field $\mathbb{Q}(\omega)$ there are decompositions of~$-1$ into the sums of two squares, for example: $$\Bigl(\frac{\sqrt{-7}+\sqrt{5}}{-1+\sqrt{-35}}\Bigl)^2 + \Bigl(\frac{6}{-1 + \sqrt{-35}}\Bigl)^2 = -1, \quad \omega^2 + (\omega^2)^2 = -1.$$
\end{rem}

\section{Automorphisms of conic bundles}\label{sect: conic bundles}
In this section we evaluate Jordan constants of automorphism groups of surfaces having a conic bundle structure over $\mathbb{P}^1_K$. At first, let us recall the definition. Let $X$ be a smooth surface and let $\pi:X \xrightarrow{}B$ be a dominant morphism with connected fibers to a smooth curve. Then $\pi$ is called a conic bundle if each (scheme) fiber of $\pi$ is isomorphic to a conic.

Let $X$ be a smooth surface over the field $K$, and let $\pi: X\xrightarrow{}{}\mathbb{P}^1_K$ be a conic bundle. Let $G\subset\Aut(X,\pi)$ be a finite subgroup, then it is included in a short exact sequence $$1\xrightarrow{}{} G_F\xrightarrow{}{} G\xrightarrow{p}{} G_B\xrightarrow{}{} 1.$$ Here $G_B$ is a finite subgroup in automorphism group of base, that is $G_B \subset \PGL_2(K)$. A subgroup $G_F$ acts on $X$ preserving fibers. Since $G_F$ is finite, there exist a rational fiber with faithful action of group $G_F$, that is $G_F \subset \PGL_2(K)$. We denote the homomorphism from~$G$ to $G_B$ by $p$.

Now we evaluate the Jordan constant of the group $G$ for different $G_F$ and $G_B$. We use approximately the same methods as in~\hbox{\cite[Proposition 3.2]{Yas}}, however, we need more precise estimates.

Firstly, we need the following auxiliary geometric lemma.

\begin{lemma}[{\cite[Lemma 5.2]{Serre 2}}]\label{lemma: eigenvalue}
Let $g\in G$, and $h\in G_F$ be such that $g$ normalizes the cyclic subgroup $\langle h\rangle$. Then $ghg^{-1}$ equals either $h$ or $h^{-1}$.
\end{lemma}

The following lemma follows from the proof of Proposition~\mbox{\cite[Proposition 3.2]{Yas}}. For the convenience of the reader, we write down this proof.

\begin{lemma}\label{lemma: G_F = C,D; G_B = C,D} 
Let $G_F$ be isomorphic to either $\mathbb{Z}/n\mathbb{Z}$ or $D_{2 n}$, and $G_B$ is isomorphic to either~$\mathbb{Z}/m\mathbb{Z}$ or~$D_{2m}$. Then $J(G) \leqslant 32$.
\end{lemma}

\begin{proof}
Firstly, assume that $G_F$ is not isomorphic to $D_4\simeq(\mathbb{Z}/2\mathbb{Z})^2$. Then there is an~$h\in G_F$ such that the cyclic subgroup $\langle h\rangle\subset G_F$ is characteristic of index at most~$2$. At the same time, $G_B$ contains a normal cyclic subgroup $G'_B$ of index at most~$2$. Denote~$G' =~p^{-1}(G'_B)$. Then $G'$ is a normal subgroup in $G$ of index at most $2$. Let~$g\in G'$ be such that $p(g)$ generates $G'_B$. The element $g$ normalizes the subgroup $\langle h\rangle$, then by Lemma~\ref{lemma: eigenvalue} either~$ghg^{-1}=h$, or~$ghg^{-1}=h^{-1}$. In any case, the element $g^2$ commutes with~$h$, hence the subgroup $\langle h, g^2\rangle\subset G'$ is an abelian subgroup of index at most~$4$. By Theorem~\ref{theo: weak jordan constant} there is an abelian characteristic subgroup $A$ of index at most 16 in the group~$G'$. Then~$A\subset G$ is the abelian normal subgroup of index at most 32, and~$J(G)\leqslant 32$.

Suppose $G_F$ is isomorphic to $D_4$. Again, $G_B$ contains a normal cyclic subgroup $G'_B$ of index at most~$2$. Denote~$G' =~p^{-1}(G'_B)$. Then $G'$ is a normal subgroup in $G$ of index at most $2$. Let~$g\in G'$ be such that $p(g)$ generates $G'_B$. Then subgroup $\langle g\rangle\subset G_F$ is abelian of index at most 4. By Theorem~\ref{theo: weak jordan constant} there is an abelian characteristic subgroup $A$ of index at most 16 in the group~$G'$. Then~$A\subset G$ is the abelian normal subgroup of index at most~$32$, and~$J(G)\leqslant 32$.
\end{proof}

The following three lemmas are purely group-theoretic assertions.

\begin{lemma}\label{lemma: G_F = C,D; G_B = A,S}
Let $G_F$ be isomorphic to either $\mathbb{Z}/n\mathbb{Z}$, or $D_{2n}$, and $G_B$ be isomorphic to either~$\mathfrak{A}_4$, or $\mathfrak{S}_4$. Then $J(G) \leqslant 48$.
\end{lemma}
\begin{proof}
There is a characteristic abelian subgroup $G' _F$ of index at most 2 in $G_F$ (if~\mbox{$G_F \simeq D_4$}, then the desired subgroup is the entire group $D_4$). Since~$G_F$ is normal in $G$, then $G'_F$ is also normal in $G$ and $$[G:G' _F] \leqslant 2|G_B| \leqslant 48.$$ Hence, $J(G) \leqslant 48$.
\end{proof}

\begin{lemma}\label{lemma: G_F = A; G_B = C, D, A, S}
Let $G_F$ be isomorphic to $\mathfrak{A}_4$, and $G_B$ be isomorphic to either~$\mathbb{Z}/n\mathbb{Z}$, or~$D_{2n}$, or~$\mathfrak{A}_4$, or $\mathfrak{S}_4$. Then $J(G) \leqslant 72$.
\end{lemma}

\begin{proof}
Firstly, let us consider the case when $G_B$ is isomorphic to either $\mathfrak{A}_4$ or $\mathfrak{S}_4$. Since~$G_F$ is isomorphic to $\mathfrak{A}_4$, it contains the characteristic abelian Klein subgroup $G'_F$ of index 3. Hence $G'_F$ is a normal abelian subgroup in $G$ of index $$[G:G'_F] \leqslant 3|G_B| \leqslant 72.$$ Therefore $J(G) \leqslant 72$.

Now let $G_B$ be isomorphic to $\mathbb{Z}/n\mathbb{Z}$. We have the following exact sequence: $$1 \xrightarrow{}{} \mathfrak{A}_4 \xrightarrow{}{} G\xrightarrow{}{} \mathbb{Z}/n\mathbb{Z}\xrightarrow{}{} 0.$$
In this case, according to Proposition~\ref{prop: A_4 x Zm}, there is a normal subgroup $G'$ in $G$ of index at most~$2$ and isomorphic to $\mathfrak{A}_4\times\mathbb{Z}/k\mathbb{Z}$. Denote by $\mathfrak{N}$ the Klein four-subgroup in $\mathfrak{A}_4$. Then~$\mathfrak{N}\times\mathbb{Z}/k\mathbb{Z}$ is a unique abelian subgroup of index $3$ in $G'$, therefore it is a characteristic subgroup. It follows that $G$ contains a normal abelian subgroup of index at most $6$, that is, $J(G) \leqslant 6$.

If $G_B$ is isomorphic to $D_{2n}$, then we have the following exact sequence: $$1\xrightarrow{}{}\mathfrak{A}_4\xrightarrow{}{} G\xrightarrow{}{} D_{2 n}\xrightarrow{}{} 0.$$ 
Consider the cyclic subgroup~$\mathbb{Z}/n\mathbb{Z}\subset D_{2n}$ and its preimage $\Tilde{G} = p^{-1}(\mathbb{Z}/n\mathbb{Z})$. According to Proposition~\ref{prop: 
A_4 x Zm}, there is a normal subgroup $G'$ in $\Tilde{G}$ of index at most $2$ and isomorphic to $\mathfrak{A}_4\times\mathbb{Z}/k\mathbb{Z}$. Note that $G'$ is also a normal in $G$ of index at most $4$ (as a preimage of normal). Similarly to the previous case, the group $G'$ contains a characteristic abelian subgroup of index $3$, which means that $G$ contains a normal abelian subgroup of index at most $12$, that is, $J(G) \leqslant 12$. \end{proof}

\begin{lemma}\label{lemma: G_F = S; G_B = C, D, A, S}
Let $G_F$ be isomorphic to $\mathfrak{S}_4$, and $G_B$ is isomorphic to either~$\mathbb{Z}/n\mathbb{Z}$, or $D_{2n}$, or~$\mathfrak{A}_4$, or $\mathfrak{S}_4$. Then $J(G) \leqslant 72$.
\end{lemma}

\begin{proof} 
Firstly, assume that $G_B$ is isomorphic to $\mathbb{Z}/n\mathbb{Z}$, we have the following exact sequence: $$1\xrightarrow{}{}\mathfrak{S}_4\xrightarrow{}{} G\xrightarrow{}{}\mathbb{Z}/n\mathbb{Z} \xrightarrow{}{} 0.$$
By Proposition~\ref{prop: ext of S_n}, the group $G$ is isomorphic to the direct product $\mathfrak{S}_4\times\mathbb{Z}/n\mathbb{Z}$, in particular, it contains a normal abelian subgroup $\mathfrak{N}\times\mathbb{Z}/n\mathbb{Z}$ of index $6$, that is,~$J(G) \leqslant 6$.

Secondly, assume that $G_B$ is isomorphic to $\mathfrak{A}_4$. Since $G_F$ is isomorphic to $\mathfrak{S}_4$, it contains the characteristic abelian Klein four-subgroup $G'_F$ of index 6. Hence $G'_F$ is a normal abelian subgroup in $G$ of index $$[G:G'_F] \leqslant 6|G_B| \leqslant 72.$$ Therefore $J(G) \leqslant 72$.

Now let $G_B$ be isomorphic to $D_{2n}$, then we have the following exact sequence: $$1 \xrightarrow{}{}\mathfrak{S}_4\xrightarrow{}{} G\xrightarrow{}{}D_{2n}\xrightarrow{}{} 0.$$ Let $G'\subset G$ be the preimage of the subgroup $\mathbb{Z}/n\mathbb{Z}\subset D_{2n}$. Then $G'$ is normal, since it is a preimage of normal, and, according to Proposition~\ref{prop: ext of S_n}, it is isomorphic to the direct product $\mathfrak{S}_4\times\mathbb{Z}/n\mathbb{Z}$. Then~$\mathfrak{N}\times\mathbb{Z}/k\mathbb{Z}$ is a unique normal abelian subgroup of index~$6$ in~$G'$, therefore it is a characteristic subgroup. It follows that $G$ contains a normal abelian subgroup of index at most $12$, that is, $J(G) \leqslant 12$.

If $G_B$ is isomorphic to $\mathfrak{S}_4$, then we have the following exact sequence:
$$1\xrightarrow{}{} \mathfrak{S}_4\xrightarrow{}{} G\xrightarrow{}{}\mathfrak{S}_4\xrightarrow{}{} 0.$$ Let $G'\subset G$ be a preimage of the Klein four-subgroup. Then $G'$ is normal, since it is a preimage of normal, and $[G:G'] = 6$. Since the Klein subgroup is abelian, then by Proposition~\ref{prop: ext of S_n} the group $G'$ is isomorphic to the direct product $\mathfrak{S}_4\times(\mathbb{Z}/2\mathbb{Z})^2$. Therefore,~$G'$ contains an abelian characteristic subgroup isomorphic to $(\mathbb{Z}/2\mathbb{Z})^4$, of index~$6$, hence $G$ contains a normal abelian subgroup of index~$36$, that is, $J(G)\leqslant 36$.
\end{proof}

From the previous four lemmas we obtain the following corollary.

\begin{sled}\label{sled: J(Aut(conic bundles))} 
Let $K$ be a field of characteristic $0$, such that either $\sqrt{5}\not\in K$, or~$-1$ is not a sum of two squares in $K$. Let $X$ be a smooth rational surface over the field $K$, and let~$\pi: X\xrightarrow{}{}\mathbb{P}^1_K$ be a conic bundle. Then $$J( \Aut(X,\pi)) \leqslant 72.$$
\end{sled}

\begin{proof}
Let $G$ be a finite subgroup in $J(\Aut(X,\pi))$. Then, as already discussed, the group $G$ is included in the exact sequence $$1 \xrightarrow{}{} G_F\xrightarrow{}{} G\xrightarrow{}{} G_B\xrightarrow{}{} 1,$$ where $G_F$ and $G_B$ are finite subgroups in $\PGL_2(K)$. By Proposition~\ref{prop:Beauville}, the group~$\PGL_2(K)$ can contain only the following finite subgroups: $\mathbb{Z}/n\mathbb{Z}$, $D_{2n}\;(\textit{for} \;n\geqslant 2)$, $\mathfrak{A}_4$ and $\mathfrak{S}_4$. The case when $G_F$ and $G_B$ are cyclic or dihedral groups is handled by Lemma~\ref{lemma: G_F = C,D; G_B = C,D}. The case when $G_F$ is cyclic or dihedral, and $G_B$ is isomorphic to either $\mathfrak{A}_4$ or $\mathfrak{S}_4$ is considered in Lemma~\ref{lemma: G_F = C,D; G_B = A,S}. Finally, the cases when $G_F$ is isomorphic to either $\mathfrak{A}_4$ or $\mathfrak{S}_4$ are handled by  Lemmas~\ref{lemma: G_F = A; G_B = C, D, A, S} and~\ref{lemma: G_F = S; G_B = C, D, A, S}. Thus~$J(G)$ does not exceed $72$ in all possible cases, hence we have the inequality~$J(\Aut(X,\pi))) \leqslant 72$. \end{proof}

\section{Del Pezzo surfaces}\label{sect: del Pezzo}

In this section we estimate Jordan constants of automorphism groups of rational del Pezzo surfaces. Recall that a \textit{del Pezzo surface} is a smooth projective surface with an ample anticanonical class. The self-intersection index of the canonical class is called the degree of a del Pezzo surface. Also recall that the degree of the del Pezzo surface can take integer values from $1$ to~$9$.

First of all, we need a theorem about the Jordan constants of the automorphism group of the rational del Pezzo surface of degree 9, that is, the projective plane.

\begin{theorem}[{\cite[Theorem 1.2]{Hu}}, weak version]\label{theo: J(GL)}
Let $K$ be a field of characteristic $0$. Then the Jordan constant of the group $\PGL_3(K)$ takes the following values:
\begin{enumerate}

\item $J(\PGL_3(K)) = 360$ if and only if $\omega \in K$ and $\sqrt{5}\in K$;
\item $J(\PGL_3(K)) = 168$ if and only if $\sqrt{-7}\in K$, and at the same time either~$\omega\not\in K$, or $\sqrt{5}\not\in K$;
\item$J(\PGL_3(K)) \leqslant 60$ if and only if $\sqrt{-7}\not\in K$, and at the same time either ~$\omega\not\in K$, or $\sqrt{5} \not\in K$.
\end{enumerate}
\end{theorem}

We now proceed to evaluate the Jordan constants of automorphism groups of rational del Pezzo surfaces. But it is worth noting that rationality is not used in the following lemma.

\begin{lemma}\label{lemma: dP cases; not P1xP1}
Let $K$ be a field of characteristic $0$, and let $X$ be a del Pezzo surface of degree $d \neq 2,8,9$ over $K$. Then $$J(\Aut(S)) \leqslant 120.$$
\end{lemma} 

\begin{proof}
We proceed case by case. 




Let $d = 7$. In this case, there is $(-1)$-curve on $S$ defined over~$K$ and invariant with respect to all automorphisms. Therefore, there exists an $\Aut(S)$-equivariant contraction on a smooth surface $S'$. Thus, the group $\Aut(S)$ effectively acts on $S'$ with a fixed point (the image of the contracted $(-1)$-curve), and by Corollary~\ref{lemma: action with fixed point} we have~$J(\Aut(S)) \leqslant 60$.

Let $d = 6$. Then by Theorem~\cite[Theorem 8.4.2]{D} we have the isomorphism $$\Aut\left(\overline{S}\right) \simeq \left(\overline{K}^\times\right)^2 \rtimes D_{12}.$$ Hence, $J(\Aut(S)) \leqslant J(\Aut\left(\overline{S}\right)) \leqslant 12$.

Let $d = 5$. Then by Theorem~\cite[Theorem 8.5.6]{D} we have the isomorphism $$\Aut\left(\overline{S}\right) \simeq\mathfrak{S}_5.$$ Therefore, $J(\Aut(S)) \leqslant J(\Aut\left(\overline{S}\right)) = 120$.

Let $d = 4$. Then by Theorem~\cite[Theorem 8.6.6]{D} we have an embedding $$\Aut\left(\overline{S}\right) \subset (\mathbb{Z}/2\mathbb{Z})^4\rtimes\mathfrak{S}_5.$$ Hence, $J(\Aut(S))\leqslant J(\Aut\left(\overline{S}\right)) \leqslant 120$ (one can prove a more precise estimate, but for our purposes, this is enough).

Let $d = 3$. In this case $S$ is a cubic surface in $\mathbb{P}^3_K$. From the classification (see for example~\cite[Theorem 9.5.6]{D}) it can be seen that the group $\Aut(S)$ is either isomorphic to~$(\mathbb{Z}/3\mathbb{Z})^3\rtimes\mathfrak{S}_4$ (Fermat cubic case), or has at most $120$ elements. In any case, we have the inequality $J(\Aut(S)) \leqslant 120$.


Let $d = 1$. In this case there is a rational point on $S$ that is invariant with respect to all automorphisms, namely, the intersection point of all divisors from the complete linear system $|-K_S|$. Hence, by Corollary~\ref{lemma: action with fixed point} we have the inequality $J(\Aut(S)) \leqslant 60$. \end{proof}

After we have received the estimates, it is convenient to note the following.

\begin{lemma}\label{lemma: dP 2}
Let $K$ be a field of characteristic $0$. Let $S$ be a del Pezzo surface of degree~$2$ over $K$. Then either $J(\Aut(S)) \leqslant J(\PGL_3(K))$, or $J(\Aut(S)) \leqslant 96.$
\end{lemma}

\begin{proof}
The anticanonical linear system defines a double-covering structure over $\mathbb{P}^2_K$ on~$S$ with branching over a smooth quartic $C$. Therefore, the automorphism group of the surface $S$ is included in the following exact sequence $$0\xrightarrow{}{}\mathbb{Z}/2\mathbb{Z}\xrightarrow{}{}\Aut(S)\xrightarrow{}{}\Aut(C)\xrightarrow{}{} 1.$$ From this sequence we immediately obtain an estimation $J(\Aut(S)) \leqslant |\Aut(C)|$. Automorphism groups of smooth flat quartics over $\overline{K}$ are described in~\cite[Theorem 6.5.2]{D}. It follows from this theorem that if the automorphism group $\Aut(\overline{C})$ is not isomorphic to~$\text{PSL}_2(\mathbb{F}_7)$, then it has no more than $96$ elements. But the group $\Aut(C)$ is a subgroup in~$\Aut(\overline{C})$, so if $|\Aut(\overline{C})| \leqslant 96$, then $|\Aut(C)| \leqslant 96$, and in this case the lemma is proved.

Suppose $$\Aut(\overline{C}) \simeq\text{PSL}_2(\mathbb{F}_7),$$ in particular $|\Aut(\overline{C})| = 168$. If~$\Aut(C)\subset\Aut(\overline{C})$ is a proper subgroup, then we have the inequality $|\Aut(C)|\leqslant 56$, and in this case the lemma is also proved. So we can assume that $\Aut(C) \simeq\text{PSL}_2(\mathbb{F}_7)$. Since the embedding of the curve $C$ into the projective plane is canonical, the group~$\Aut(C)$ is a subgroup in~$\PGL_3(K)$. Since the group $\text{PSL}_2(\mathbb{F}_7)$ is simple, then $$J(\PGL_3(K))\geqslant J(\text{PSL}_2(\mathbb{F}_7)) =|\text{PSL}_2(\mathbb{F}_7)|= |\Aut(C)|\geqslant J(\Aut(S)),$$ and in this case the lemma is also proved.
\end{proof}

The following proposition restricts the Jordan constant of automorphism groups of rational quadrics, which are different from the product of two projective lines.

\begin{propos}[{\cite[Proposition 1.6]{Zai}}]\label{prop: J(AutS) <= 120 if rkPic = 1}
    Let $K$ be a field of characteristic $0$. Let $S$ be a smooth rational quadric in $\mathbb{P}^3_K$, and $S\not\simeq\mathbb{P}^1_K\times\mathbb{P}^1_K$. Then $$J(\Aut(S)) \leqslant 120.$$
\end{propos}

\begin{sled}\label{sled: dP 8}
Let $K$ be a field of characteristic $0$, and let $S$ be a rational del Pezzo surface of degree $8$ over $K$. Assume that $S\not\simeq\mathbb{P}^1_K\times\mathbb{P}^1_K$. Then $$J(\Aut(S)) \leqslant 120.$$
\end{sled}

\begin{proof} 
It is known that $\overline{S}$ is isomorphic to either $\mathbb{P}_{\overline{K}}^1\times\mathbb{P}_{\overline{K}}^1$, or to the blow-up of $\mathbb{P}^2_{\overline{K}}$ at a point. If $\overline{S}$ is isomorphic to $\mathbb{P}_{\overline{K}}^1\times\mathbb{P}_{\overline{K}}^1$, then since $S$ is rational, it is isomorphic to a smooth rational quadric in~$\mathbb{P}^3_K$ (see for example~\cite[Proposition 6.1]{Zai}). According to Proposition~\ref{prop: J(AutS) <= 120 if rkPic = 1} we have $J(\Aut(S)) \leqslant 120$.

If $\overline{S}$ is isomorphic to the blow-up of $\mathbb{P}^2_{\overline{K}}$ at a point, then there is a unique $(-1)$-curve on~$S$ defined over $K$ and invariant with respect to all automorphisms. Therefore, there exists an $\Aut(S)$-equivariant contraction to the projective plane $\mathbb{P}_K^2$. Thus, the group $\Aut(S)$ effectively acts on $\mathbb{P}_K^2$ with a fixed point (the image of the contractible $(-1)$-curve), and by Corollary~\ref{lemma: action with fixed point} we have ~$J(\Aut(S)) \leqslant 60$.\end{proof}

For the quadric $\mathbb{P}^1_K\times\mathbb{P}^1_K$, there is the following result.

\begin{theorem}[{\cite[Theorem 1.4]{Zai}}]\label{theo:theo P^1 x P^1}
Let $K$ be a field of characteristic $0$.
\begin{enumerate}
    \item $J(\Aut(\mathbb{P}^1_K \times \mathbb{P}^1_K)) = 7200$ if and only if \mbox{$\sqrt{5} \in K$} and $-1$ is a sum of two squares in $K$;
    \item $J(\Aut(\mathbb{P}^1_K \times \mathbb{P}^1_K)) = 72$ if and only if \mbox{$\sqrt{5} \not\in K$}, and $-1$ is a sum of two squares in~$K$;
    \item $J(\Aut(\mathbb{P}^1_K \times \mathbb{P}^1_K)) = 8$ if and only if $-1$ is not a sum of two squares in $K$.

\end{enumerate}
\end{theorem}

For convenience, we introduce the following notation. Let $K$ be a field of characteristic~$0$. Denote by $J_{dP}(K)$ the following value: \begin{equation}\label{eq Jdp}
J_{dP}(K) = \max\limits_{X}\bigl(J(\Aut(X))\bigr),
\end{equation}
where $X$ runs through all rational del Pezzo surfaces over $K$. For this value, there is an obvious lower bound.

\begin{rem}\label{rem: max dP >= 120}
    There exists a (rational) del Pezzo surface of degree $5$ with an automorphism group isomorphic to $\mathfrak{S}_5$ over any field, see for example~\mbox{\cite[Example 3.1, Corollary 8.3]{Zai 2}}. Therefore, we have $J_{dP}(K)\geqslant 120.$
\end{rem}

Now let us gather all the statements of this section into a proposition that will be convenient to refer to.

\begin{propos}\label{prop: all Dp}
Let $K$ be a field of characteristic $0$. Then
\begin{enumerate}
        \item $J_{dP}(K) = 7200$ if and only if $\sqrt{5}\in K$, and $-1$ is a sum of two squares in $K$;
        \item $J_{dP}(K) = 168$ if and only if $\sqrt{5} \not\in K$ and $\sqrt{-7}\in K$;
        \item $J_{dP}(K) = 120$ if and only if one of the two conditions is satisfied:
        \begin{itemize}
            \item$\sqrt{5}\in K$, and $-1$ is not a sum of two squares in $K$; 
            \item $\sqrt{5} \not\in K$ and $\sqrt{-7} \not\in K$.
            \end{itemize}
    \end{enumerate}
\end{propos}

\begin{proof}
To calculate the value of $J_{dP}(K)$, one can ignore del Pezzo surfaces with Jordan constant of the automorphism group less or equal to 120 by Remark~\ref{rem: max dP >= 120}. These are all del Pezzo surfaces of degree $d\neq 2, 8, 9$ by Lemma~\ref{lemma: dP cases; not P1xP1} and rational del Pezzo surfaces of degree $8$ that are not isomorphic to $\mathbb{P}^1_K\times\mathbb{P}^1_K$, by Corollary~\ref{sled: dP 8}.

Let $X$ be a rational del Pezzo surface of degree $9$, then $X$ is isomorphic to~$\mathbb{P}^2_K$. Indeed, since $X$ is rational, there is a $K$-point on $X$ by the Lang--Nishimura theorem (see for example ~\mbox{\cite[Theorem 3.6.11]{Poo}}). Hence $X$ is a Severi--Brauer surface containing $K$-point. By Chatelet's theorem (see for example ~\mbox{\cite[Theorem 5.1.3]{Phil Gille}}), we obtain that $X$ is isomorphic to $\mathbb{P}^2_K$. In particular, the group~$\Aut(X)$ is isomorphic to the group $\PGL_3(K)$. Therefore, by Lemma~\ref{lemma: dP 2}, del Pezzo surfaces of degree $2$ can also be ignored when computing the value of~$J_{dP}(K)$. Therefore, the value of~$J_{dP}(K)$ can be rewritten as follows: $$J_{dP}(K) = \max(J(\PGL_3(K)), J(\Aut(\mathbb{P}^1_K \times\mathbb{P}^1_K)), 120).$$

Let $\sqrt{5}\in K$, and $-1$ is a sum of two squares in $K$. Then by Theorem~\ref{theo:theo P^1 x P^1} we have~$J(\Aut(\mathbb{P}^1_K\times\mathbb{P}^1_K)) = 7200$, at the same time~$J(\PGL_3(K)) \leqslant 360$ by Theorem~\ref{theo: J(GL)}. Hence~$J_{dP}(K) = 7200$.

Let $\sqrt{5}\in K$, but $-1$ is not a sum of two squares in $K$. 
Note that in this case neither~$\sqrt{-7}$ nor $\omega$ can lie in $K$ due to Remark~\ref{rem: -1 = a^2 + b^2 in special fields }. Then we have the inequality~$J(\PGL_3(K)) \leqslant 60$ by Theorem~\ref{theo: J(GL)}, and we have equality~$J(\Aut(\mathbb{P}^1_K \times \mathbb{P}^1_K)) = 8$ by Theorem~\ref{theo:theo P^1 x P^1}. Therefore~$J_{dP}(K) = 120$.

Let $\sqrt{5} \not\in K$ and $\sqrt{-7}\in K$. Then we have the inequality $J(\Aut(\mathbb{P}^1_K\times\mathbb{P}^1_K)) \leqslant 72$ by Theorem~\ref{theo:theo P^1 x P^1}, and the equality $J(\PGL_3(K)) = 168$ by Theorem ~\ref{theo: J(GL)}. Therefore~\mbox{$J_{dP}(K) = 168$}.

Let $\sqrt{5} \not\in K$ and $\sqrt{-7}\not\in K$. Then ~$J(\Aut(\mathbb{P}^1_K\times\mathbb{P}^1_K)) \leqslant 72$ by Theorem~\ref{theo:theo P^1 x P^1}, and~$J(\PGL_3(K)) \leqslant 60$ by Theorem ~\ref{theo: J(GL)}. Therefore~$J_{dP}(K) = 120$.

Thus, we have considered all possible fields of characteristic $0$, and the theorem is proved.\end{proof}

\section{Jordan constants of Cremona group of rank 2}\label{sect: Jordan constant of Cremona}

In this section we prove the main theorem and its corollary. But before that, let us make a very simple remark.

\begin{rem}\label{rem: J(Cr)<=7200}
Let $K$ be a field of characteristic $0$. Since the group $\Cr_2(K)$ is embedded in the group~$\Cr_2(\overline{K})$, then applying Theorem~\ref{theo: Egor}, we obtain an estimate $$J(\Cr_2(K)) \leqslant J(\Cr_2(\overline{K})) = 7200.$$
\end{rem}

Now everything is prepared to prove the main theorem.

\begin{proof}[of Theorem~\ref{theo: J(Cr_2)}]
Let $G$ be a finite subgroup in the group $\Cr_2(K)$. Regularizing the action of $G$ and taking $G$-equivariant resolution of singularities, we can assume that $G$ is a subgroup in the automorphism group of a smooth rational surface~$S$ (see~\mbox{\cite[Lemma 3.5]{D-I}}). In this case, $S$ contains a $K$-point according to the Lang--Nishimura theorem (see for example~\mbox{\cite[Theorem 3.6.11]{Poo}}). Further, applying the program of minimal models with an action of the group $G$, we can assume that $S$ is either a del Pezzo surface, or has a conic bundle structure $\pi: S\xrightarrow{}{} B$, and~$G\subset\Aut(S,\pi)$, see~\cite[Theorem 1G]{Isk}. Note that in the latter case, the curve $B$ is geometrically rational and contains a $K$-point, since $S$ contains a $K$-point. Hence, we have an isomorphism $B\simeq\mathbb{P}^1_K$.

Hence, the Jordan constant of the Cremona group of rank 2 can be rewritten in the following form:
$$J(\Cr_2(K)) = \max(J_{dP}(K), J_{cb}(K)).$$
Here $J_{dP}(K)$ is a constant, defined in~\eqref{eq Jdp}, and $J_{cb}(K)$ is a similar constant for conic bundles: $$J_{cb}(K) = \max\limits_{(Y,\pi_Y)}(J(\Aut(Y,\pi_Y))),$$ where pairs of $(Y,\pi_Y)$ run through all rational surfaces over $K$ together with a conic bundle structure over $\mathbb{P}^1_K$.

Note that for any field $K$ of characteristic zero, the inequality ~$J_{cb}(K)\leqslant J_{dP}(K)$ is satisfied. Indeed, for fields~$K$ in which either $\sqrt{5}$ does not lie, or $-1$ is not a sum of two squares we have $$J_{cb}(K) \leqslant 72 < 120\leqslant J_{dP}(K)$$ by Corollary~\ref{sled: J(Aut(conic bundles))} and Remark~\ref{rem: max dP >= 120}. If $\sqrt{5}\in K$, and $-1$ is a sum of two squares in $K$, then~$J_{dP}(K) = 7200$, and $J_{cb}(K)$ cannot exceed $7200$ by Remark~\ref{rem: J(Cr)<=7200}.

Now the assertion of the theorem directly follows from Proposition~\ref{prop: all Dp}.
\end{proof}

To prove Corollary~\ref{sled: Cr(k), Cr(R), Cr(Q)...} we prove an elementary auxiliary lemma.

\begin{lemma}\label{lemma:support}
Let $m$ and $n$ be two coprime square-free integers. Then $\sqrt{m}$ does not lie in the field $\mathbb{Q}(\sqrt{n})$.
\end{lemma}

\begin{proof} Since $m$ and $n$ are square-free integers, the numbers $\sqrt{m}$ and $\sqrt{n}$ are irrational. Suppose that $\sqrt{m}$ lies in $\mathbb{Q}(\sqrt{n})$. Then there are $a, b, c, d\in \mathbb{Z}$ such that $$\Bigl(\frac{a}{b} + \frac{c}{d}\sqrt{n}\Bigl)^2 = m,$$ and we can assume that fractions are irreducible. Let us act by Galois involution of quadratic extension $\mathbb{Q} \subset \mathbb{Q}(\sqrt{n})$ on both sides of the equality. We obtain: 
$$\Bigl(\frac{a}{b} - \frac{c}{d}\sqrt{n}\Bigl)^2 = m.$$
It follows form this two equalities that either $a = 0$, or~$c = 0$. If~$a = 0$, we get the equality~$nc^2 = md^2$, therefore,~$c$ is divisible by $m$, but then $d$ is divisible by $m$, and we got a contradiction with the irreducibility of the fraction $\frac{c}{d}$. If we assume that $c = 0$, then we immediately get a contradiction with the irrationality of the number $\sqrt{m}$. Thus, we got contradictions in all cases, so $\sqrt{m}\notin\mathbb{Q}(\sqrt{n})$. \end{proof}

\begin{proof}[of Corollary~\ref{sled: Cr(k), Cr(R), Cr(Q)...}.] 
For fields $\mathbb{C}$, $\mathbb{R}$ and $\mathbb{Q}$, the assertion is guaranteed by Theorem~\ref{theo: Egor} and this is consistent with Theorem~\ref{theo: J(Cr_2)}. Applying Remark~\ref{rem: -1 = a^2 + b^2 in special fields } and Theorem~\ref{theo: J(Cr_2)}, we get $$J(\Cr_2(\mathbb{Q}(\sqrt{5}, \sqrt{-7}))) = J(\Cr_2(\mathbb{Q}(\sqrt{5}, \omega))) = 7200.$$ By Lemma~\ref{lemma:support}, the number $\sqrt{5}$ does not lie in $\mathbb{Q}(\sqrt{-7})$. So, by Theorem~\ref{theo: J(Cr_2)} we have $$J(\Cr_2(\mathbb{Q}(\sqrt{-7}))) = 168.$$ Finally, since $-1$ is not a sum of two squares in $\mathbb{Q}(\sqrt{5})$, then by Theorem~\ref{theo: J(Cr_2)} we have $$J(\Cr_2(\mathbb{Q}(\sqrt{5}))) = 120.$$
\end{proof}

\end{document}